\documentclass[reqno,12pt]{amsart}
\usepackage{hyperref}
\usepackage{amsmath,amsthm}
\usepackage{amssymb}
\usepackage{mathabx}
\usepackage[latin1]{inputenc}

\newtheorem{Theorem}{Theorem } [section]
\newtheorem{lemma}[Theorem]{Lemma}
\newtheorem{remark}{Remark}

\numberwithin{equation}{section}

\newcommand{\veee}{{\,\widecheck{\;}\,}}

\font\ff=cmsy10
\def\tiF{\text{\ff F\kern 0pt}{\;}^{ -1}}
\def\tF{\text{\ff F\kern 0pt}}
\begin{document}
\title[]{The Cauchy problem for a fifth order KdV equation\\ in weighted Sobolev spaces}
\author{Eddye Bustamante, Jos\'e Jim\'enez and Jorge Mej\'{\i}a}

\address{Eddye Bustamante M., Jos\'e Jim\'enez U., Jorge Mej\'{\i}a L. \newline
Departamento de Matem\'aticas\\Universidad Nacional de Colombia\newline
A. A. 3840 Medell\'{\i}n, Colombia}
\email{eabusta0@unal.edu.co, jmjimene@unal.edu.co, jemejia@unal.edu.co}

\begin{abstract}
In this work we study the initial value problem (IVP) for the fifth order KdV equations,
\begin{align*}
\partial_{t}u+\partial_{x}^{5}u+u^k\partial_{x}u=0,\text{} & \quad x,t\in \mathbb R, \quad k=1,2,\\
\end{align*}
in weighted Sobolev spaces $H^s(\mathbb R)\cap L^2(\langle x \rangle^{2r}dx)$. We prove local and global results. In the case $k=2$ we point out the relation between decay and regularity of the solution of the IVP.
\end{abstract}

\maketitle

\section{Introduction}\label{int}

In this article we will consider the IVP
\begin{align}
\left. \begin{array}{rr}
\partial_{t}u+\partial_{x}^{5}u+u^k\partial_{x}u=0,\text{} & \quad x,t\in \mathbb R\\
u(0)=u_0 & 
\end{array}
\right\} ,\label{KdV}
\end{align}

with $k=1,2$. When $k=1$ we will refer to this problem as the IVP for the fifth order KdV equation. When $k=2$ we will refer to this problem as the IVP for the modified fifth order KdV equation.

When $k=1$ the equation was proposed by Kakutani and Ono as a model for magneto-acoustic waves in plasma physics (see \cite{KO}).

The equations that we will study are included in the class
\begin{align}\label{horder}
\partial_t u+\partial_x^{2j+1} u+P(u,\partial_x u,\cdots, \partial_x^{2j}u)=0,\quad x,t\in\mathbb R,\quad j\in \mathbb Z^+,
\end{align}
where $P:\mathbb R^{2j+1}\to\mathbb R$ (or $P: \mathbb C^{2j+1}\to\mathbb C$) is a polynomial having no constant or linear terms, i.e.
$$P(z)=\sum_{|\alpha|=l_0}^{l_1} a_\alpha z^\alpha\text{ with $l_0\geq 2$ and $z=(z_1,\cdots, z_{2j+1})$}.$$
The class in \eqref{horder} generalizes several models, arising in both mathematics and physics, of higher-order nonlinear dispersive equations.

During many years the well-posedness of these IVP has been studied in the context of the classical Sobolev spaces $H^s(\mathbb R)$. In particular, fifth order KdV equations with more general non-linearities, than those we are considering, were studied in \cite{P}, \cite{L}, \cite{Kw}, \cite{Kw1}, \cite{GKK} and \cite{KP}.

In 1983 Kato, in \cite{K}, studied the IVP for the generalized KdV equation in several spaces, besides the classical Sobolev spaces. Among them, Kato considered weighted Sobolev spaces.

In this work we will be concerned with the well-posedness of the IVP \eqref{KdV} in weighted Sobolev spaces. This type of spaces arises in a natural manner when we are interested in determining if the Schwartz space is preserved by the flow of the evolution equation in \eqref{KdV}.

In \cite{KPVHigher} Kenig, Ponce and Vega studied the IVP associated to the equation \eqref{horder} in weighted Sobolev spaces $H^s(\mathbb R)\cap L^2(|x|^m dx)$, with $m$ positive integer.

Some relevant nonlinear evolution equations as the KdV equation, the non-linear Schrödinger equation and the Benjamin-Ono equation, have also been studied in the context of weighted Sobolev Spaces (see \cite{N}, \cite{NP1}, \cite{NP2}, \cite{IO}, \cite{IO1}, \cite{FP}, \cite{FLP1}, \cite{FLP2}, \cite{FLP}, \cite{J} and \cite{CP} and references therein).

We will study real valued solutions of the IVP \eqref{KdV} in the weighted Sobolev spaces
$$Z_{s,r}:=H^s(\mathbb R)\cap L^2(\langle x\rangle^{2r} dx),$$
where $\langle x \rangle:=(1+x^2)^{1/2}$, and $s,r\in\mathbb R$.

The relation between the indices $s$ and $r$ for the IVP \eqref{KdV} can be found, after the following considerations, contained in the work of Kato:\\
Suppose we have a solution $u\in C([0,\infty); H^s(\mathbb R))$ to the IVP \eqref{KdV} for some $s\geq 2$. We want to estimate $(pu,u)$, where $p(x):=\langle x\rangle^{2r}$ and $(\cdot,\cdot)$ is the inner product in $L^2(\mathbb R)$. Proceeding formally we multiply the equation in \eqref{KdV} by $up$, integrate over $x\in\mathbb R$ and apply integration by parts to obtain:
\begin{align}
\dfrac{d}{dt} (pu,u)=5(p^{(1)} \partial_x^2 u, \partial_x^2 u)-5(p^{(3)}\partial_x u,\partial_x u)+(p^{(5)}u,u)+\dfrac 2{k+2}(p^{(1)} u^{k+2},1).\label{p}
\end{align}
To see that $(pu,u)$ is finite and bounded in $t$, we must bound the right-hand side in \eqref{p} in terms of $(pu,u)$ and $\| u\|_{H^s}^2$. The most difficult term to control in the right-hand side in \eqref{p} is $5(p^{(1)}\partial_x^2 u,\partial_x^2 u)$. Using the interpolation Lemma \ref{uno} (see section 2), for $\theta\in [0,1]$ and $u\in Z_{s,r}$ we have
$$\| \langle x \rangle^{(1-\theta)r} u \|_{H^{\theta s}}\leq C \| \langle x \rangle^r u \|_{L^2}^{1-\theta}\|u \|_{H^s}^\theta.$$
The term $5(p^{(1)}\partial_x^2 u,\partial_x^2 u)$ can be controlled when $\theta s=2$ if $p^{(1)}(x)\sim \langle x \rangle^{2(1-\theta)r}$. Taking into account that $p^{(1)}(x)\sim \langle x\rangle^{2r-1}$, we must require that $2r-1=2(1-\theta)r$ and $\theta s=2$, which leads to $s=4r$. In this way the natural weighted Sobolev space to study the IVP \eqref{KdV} is $Z_{4r,r}$.

Now, we will describe the main results of this work.

With respect to the IVP \eqref{KdV} with $k=1$ we establish local well-posedness (LWP) in $Z_{4r,r}$ for $\frac 5{16}<r<\frac 12$ and global well-posedness (GWP) in $Z_{4r,r}$, for $r\geq \frac 12$.

In the first case ($\frac 5{16}<r<\frac 12$), we use the known linear estimates for the group associated to the linear part of the equation, which were obtained by Kenig, Ponce and Vega in \cite{KPVOsc}, \cite{KPVWell1} and \cite{KPVWell}, and a pointwise formula for the group, related with fractional weights, which was deduced by Fonseca, Linares and Ponce in \cite{FLP}. These ingredients allow us to use a contraction principle in an adequate subspace of $C([0,T]; Z_{4r,r})$ to the integral equation associated to our IVP, to prove LWP in $Z_{4r,r}$.

In the second case ($r\geq\frac 12$) we use the LWP of the IVP \eqref{KdV} in the context of the Sobolev spaces $H^{4r}(\mathbb R)$, which can be obtained in a similar fashion, as it was done by Kenig, Ponce and Vega in \cite{KPVWell1} and \cite{KPVWell} for the KdV equation, to get a solution $u\in C([0,T];H^{4r}(\mathbb R))$. Then we perform a priori estimates on the the differential equation in order to prove that if the initial data belongs to $H^{4r}(\mathbb R)\cap L^2(\langle x \rangle^{2r} dx )$ then necessarily $u\in L^\infty([0,T];L^2(\langle x \rangle^{2r} dx)$. In this step of the proof we apply the interpolation inequality (Lemma \ref{uno} in section 2), mentioned before, which was proved in \cite{FP}. Finally, we give the proof of the continuous dependence of the solution on the initial data in $Z_{4r,r}$.

With respect to the IVP \eqref{KdV} with $k=2$, we establish LWP and GWP in $Z_{2,1/2}$. For the LWP, again, the idea of the proof is to apply the contraction principle to the integral equation associated to the IVP, in a certain subspace of $C([0,T];H^2(\mathbb R))$, in which we consider additional mixed space-time norms, suggested by the linear estimates of the group. This way, we obtain, firstly, a solution in $C([0,T];H^2(\mathbb R))$. Then, proceeding as in the IVP \eqref{KdV} with $k=1$, in the case $r\geq 1/2$, we can affirm that $u\in C([0,T];Z_{2,1/2})$ and that the IVP \eqref{KdV} with $k=2$ is LWP in $Z_{2,1/2}$.

To deduce GWP results from LWP results we use the following conservation laws for the solutions of the IVP \eqref{KdV} (see \cite{KPVOsc}):
\begin{align}
I_1(t)&:=\int_\mathbb R u^2(t) dx=I_1(0), \text{ for } k=1,2,\label{cl1}\\
I_2^1(t)&:=\dfrac 16 \int_\mathbb R u^3(t) dx+\dfrac 12 \int_\mathbb R (\partial_x^2 u)^2(t) dx=I_2^1(0), \text{ for }k=1, \text{ and},\label{cl2}\\
I_2^2(t)&:=\dfrac 1{12} \int_\mathbb R u^4(t) dx+ \int_\mathbb R (\partial_x^2 u)^2(t) dx=I_2^2(0), \text{ for } k=2.\label{mcl2}
\end{align}
In \cite{ILP}, Isaza, Linares and Ponce showed that there exists a relation between decay and regularity for the solutions of the KdV equation in $L^2(\mathbb R)$. More precisely, they proved that if $u\in C(\mathbb R;L^2(\mathbb R))$ is the global solution of the equation
$$\partial_t u+\partial_x^3 u+u\partial_x u=0,$$
obtained in the context of the Bourgain spaces (see \cite{KPVBil}), and there exists $\alpha>0$ such that in two different times $t_0,t_1\in\mathbb R$
$$|x|^\alpha u(t_0), |x|^\alpha u(t_1)\in L^2(\mathbb R),$$
then $u\in C(\mathbb R,H^{2\alpha}(\mathbb R))$. To achieve this goal, they chose a functional setting, where the norm $\|\partial_x u \|_{L^\infty(\mathbb R; L^2([0,T]))}$ of the solution $u$ depends continuously on the initial data in $L^2(\mathbb R)$.

Following the method used in \cite{ILP}, and taking into account that the norm $\|\partial_x^4 u \|_{L^\infty(\mathbb R; L^2([0,T]))}$ of the solution $u$ of the IVP \eqref{KdV} with $k=2$, depends continuously on the initial data in $Z_{2,1/2}$, we prove that if $u\in C([0,T];Z_{2,1/2})$ is a solution of the IVP \eqref{KdV} with $k=2$ and, for some $\alpha \in (0,1/8]$, there exist two different times $t_0,t_1\in[0,T]$ such that $|x|^{1/2+\alpha} u(t_0)$ and $|x|^{1/2+\alpha} u(t_1)$ are in $L^2(\mathbb R)$ then $u\in C([0,T];H^{2+4\alpha}(\mathbb R))$.

Before stating in a precise manner the main results of this article, let us explain the notation for mixed space-time norms. For $f:\mathbb R\times [0,T]\to\mathbb R$ (or $\mathbb C$) we have
$$\|f \|_{L^p_x L^q_T}:=\left( \int_{\mathbb R} \left( \int_0^T |f(x,t)|^q dt\right)^{p/q} dx \right )^{1/p}.$$
When $p=\infty$ or $q=\infty$ we must do the obvious changes with \textit{essup}. Besides, when in the space-time norm appears $t$ instead of $T$, the time interval is $[0,+\infty)$.

Our results are as follows:
\begin{Theorem}\label{prin1}
Let  $r>\frac{5}{16}$ and $u_0\in Z_{4r,r}$. Then there exist $T>0$ and a unique $u$, solution of the IVP \eqref{KdV} with $k=1$ such that
\begin{align}
u\in C([0,T];&Z_{4r,r}),\label{con1}\\
\| \partial_x u \|_{L^4_T L^\infty_x}&<\infty,\label{con2}\\
\|D_x^{4r}\partial_x u\|_{L_x^{\infty}L_T^2}&<\infty,\quad\text{and}\label{con3}\\
\|u\|_{L_x^2L_T^{\infty}}&<\infty.\label{H1}
\end{align}
Moreover, for any $T'\in(0,T)$ there exists a neighborhood $V$ of $u_0$ in $Z_{4r,r}$ such that the data-solution map $\tilde{u}_0\mapsto\tilde{u}$ from $V$ into the class defined by \eqref{con1}-\eqref{H1} with $T'$ instead of $T$ is Lipschitz.

When $5/16<r<1/2$, $T$ depends on $\|u_0 \|_{Z_{4r,r}}$, and when $r\geq 1/2$ the size of $T$ depends only on $\| u_0\|_{H^{4r}}$.
\end{Theorem}

Let us recall that the operator $D$ is defined through the Fourier transform by the multiplier $|\xi|$.
\begin{remark}\label{R1}\text{}
\begin{enumerate}
\item[(a)] From the proof of Theorem \ref{prin1} it is clear that if the IVP \eqref{KdV} is GWP in $H^{4r}(\mathbb R)$, $r\geq\frac12$, then the IVP is also GWP in $Z_{4r,r}$.
\item[(b)] Using the regularity property in Theorem \ref{principal} it follows, from Theorem \ref{prin1}, that the IVP \eqref{KdV} is LWP in $Z_{s,r}$ for $s\geq 4r$ and $r\geq\frac12$.
\item[(c)] Let us observe that applying the same method used in the proof of Theorem \ref{prin1} it can be seen that the IVP \eqref{KdV} is LWP in $Z_{s,l}$ with $s\geq 4r$, $l\leq r$ and $r\geq 1/2$.
\end{enumerate}
\end{remark}

\begin{Theorem}\label{prin2}
Let $r\geq 1/2$ and $u_0\in Z_{4r,r}$. Then the IVP \eqref{KdV} for the fifth order KdV equation ($k=1$) is GWP in $Z_{4r,r}$.
\end{Theorem}

\begin{Theorem} \label{tmkdv1}
Let $u_0\in Z_{2,1/2}$. Then there exist $T=T(\| u_0\|_{H^2})>0$ and a unique $u$, solution of the IVP \eqref{KdV} for the modified fifth order KdV equation (k=2), such that
\begin{align}
u\in C([0,T];&Z_{2,1/2})\,,\label{mdes11}\\
\| \partial^4_x u\|_{L^\infty_x L^2_T}&<\infty\,, \label{mdes12}\\
\|  u\|_{L^{16/5}_x L^\infty_T}&<\infty\,, \label{mdes13}\\
\|  u\|_{L^4_x L^\infty_T}&<\infty. \label{mdes14}
\end{align}
Moreover, for any $T'\in(0,T)$ there exists a neighborhood $V$ of $u_0$ in $Z_{2,1/2}$ such that the data-solution map $\tilde{u}_0\mapsto\tilde{u}$ from $V$ into the class defined by \eqref{mdes11}-\eqref{mdes14} with $T'$ instead of $T$ is Lipschitz.
\end{Theorem}

\begin{Theorem}\label{tmkdv2}
The IVP \eqref{KdV} for the modified fifth order KdV equation ($k=2$) is GWP in $Z_{2,1/2}$.
\end{Theorem}

\begin{Theorem} \label{tmkdv3}
For $T>0$ let $u\in C([0,T];Z_{2,1/2})$ be the solution of the modified fifth order KdV equation ($k=2$), obtained in Theorems \ref{tmkdv1} and \ref{tmkdv2}. Let us suppose that for $\alpha\in(0,1/8]$ there exist two different times $t_0,t_1\in [0,T]$, with $t_0<t_1$, such that $|x|^{1/2+\alpha} u(t_0)$ and $|x|^{1/2+\alpha} u(t_1)$ are in $L^2(\mathbb R)$. Then $u\in C([0,T];H^{2+4\alpha}(\mathbb R))$.
\end{Theorem}

This article is organized as follows: in section 2 we recall some linear estimates of the group associated to the linear part of the equation in the IVP \eqref{KdV}, a pointwise estimate for this group, related with fractional weights, and an interpolation inequality in weighted Sobolev spaces. In section 3 we study the IVP \eqref{KdV} with $k=1$ and prove Theorems \ref{prin1} and \ref{prin2}. In section 4 we consider the IVP \eqref{KdV} with $k=2$ and establish Theorems \ref{tmkdv1} and \ref{tmkdv2}.  In section 5 we give the proof of Theorem \ref{tmkdv3}.

Throughout the paper the letter $C$ will denote diverse constants, which may change from line to line, and whose dependence on certain parameters is clearly established in all cases.

\section{Preliminary Results}\label{pre}

In this section we recall some linear estimates for the group associated to the linear part of the equation in the IVP \eqref{KdV}, a pointwise estimate for ``fractional weights", and an interpolation inequality in weighted Sobolev spaces. On the other hand, we establish an standard estimate in weighted Sobolev spaces.

Let us consider the linear problem associated to the IVP \eqref{KdV}:
\begin{align}
\left. \begin{array}{rr}
\partial_{t}u+\partial_{x}^{5}u=0,\text{} & \quad x,t\in \mathbb R\\
u(0)=u_0 & 
\end{array}
\right\} ,\label{lKdV}
\end{align}

whose solution is given by the group $\{W(t) \}_{t\in\mathbb R}$, i.e.
$$u(x,t)=[W(t)u_0](x):=(S_t\ast u_0)(x),$$
where $S_t(x)$ is defined by the oscillatory integral
$$S_t(x)=C\int_{\mathbb R} e^{ix\xi} e^{-it\xi^5} d\xi.$$

In \cite{KPVOsc}, \cite{KPVWell1} and \cite{KPVWell}, Kenig, Ponce and Vega established the following estimates for the group $\{W(t)\}_{t\in\mathbb R}$:
\begin{enumerate}
\item[(i)] (Homogeneous smoothing effect) There exists a constant $C$ such that
\begin{align}
\|\partial_x^2 W(t) u_0 \|_{L^\infty_x L^2 _t}\leq &C\| u_0\|_{L^2}. \label{le1}
\end{align}

\item[(ii)] (Dual version of estimate \eqref{le1}) There exists a constant $C$ such that
\begin{align}
\|\partial^2_x \int_0^t W(t-t') f(\cdot,t') dt' \|_{L^\infty_T L^2_x}\leq &C \| f \|_{L^1_x L^2_T}.\label{le3}
\end{align}

\item[(iii)] (Inhomogeneous smoothing effect) There exists a constant $C$ such that
\begin{align}
\|\partial^4_x \int_0^t W(t-t') f(\cdot,t') dt' \|_{L^\infty_x L^2_t}\leq &C \| f \|_{L^1_x L^2_t}.\label{le2}
\end{align}

\item[(iv)] (Estimate of the maximal function) For any $\rho>\frac 34$ and $s>\frac 54$ there exists $C$ such that
\begin{align}
\| W(t)u_0\|_{L^2_x L^\infty_T}\leq &C(1+T)^\rho \|u_0 \|_{H^s}.\label{le4}
\end{align}

\item[(v)] There exists a constant $C$ such that, for $u_0\in H^{1/4}(\mathbb R)$,
\begin{align}
\| W(t)u_0\|_{L^4_x L^\infty_T}\leq &C\| D^{1/4} u_0\|_{L^2}, \label{le5} 
\end{align}
(see \cite{KR}).
\end{enumerate}
By interpolation it follows, from \eqref{le4} and \eqref{le5}, that for $\rho>\frac 34$ and $s>\frac 54$,
\begin{align}
\| W(t)u_0\|_{L^{16/5}_x L^\infty_T}\leq &C(1+T)^\rho \|u_0 \|_{H^s}.\label{le6m}
\end{align}

\begin{enumerate}
\item[(vi)] There exists a constant $C$ such that
\begin{align}
\| D_x^{3/4} W(t)u_0 \|_{L^4_t L^\infty_x}\leq & C\| u_0\|_{L^2}. \label{le6}
\end{align}
\end{enumerate}
Using \eqref{le1}, \eqref{le4} and \eqref{le5}, and proceeding as in the proofs of Theorem 2.1 in \cite{KPVWell} and Theorem 1.1 in \cite{KPVOsc}, it can be established the following theorem.
\begin{Theorem}\label{principal}
Let $s>\frac54$. Then for any $u_0\in H^s(\mathbb{R})$ there exist $T=T(\|u_0\|_{H^s})>0$ (with $T(\rho)\to\infty$ as $\rho\to 0$) and a unique solution $u$ of the IVP \eqref{KdV} with $k=1$, satisfying
\begin{align}
u\in C([0,T];&H^s(\mathbb{R}))\,,\label{con}\\
\| \partial_xu \|_{L^4_T L^\infty_x}&<\infty\,,\\
\|D_x^s \partial_x u \|_{L_x^{\infty}L_T^2}&<\infty\,,\quad\text{and}\label{g1d}\\
\|u\|_{L_x^2L_T^{\infty}}&<\infty\,.\label{H}
\end{align}
Moreover, for any $T'\in(0,T)$ there exists a neighborhood $V$ of $u_0$ in $H^s(\mathbb{R})$ such that the data-solution map $\tilde{u_0}\mapsto\tilde{u}$ from $V$ into the class defined by \eqref{con}-\eqref{H} with $T'$ instead of $T$ is Lipschitz. Besides, if $u_0\in H^{s'}$ with $s'>s$ then the above results hold with $s'$ instead of $s$ in the same time interval $[0,T]$ (regularity property).
\end{Theorem}

Let us observe the gain of two derivatives in $x$ in the linear estimate \eqref{le1}. However, the condition \eqref{g1d} only uses the gain of one derivative in $x$.

One of the main tools for establishing LWP of the IVP \eqref{KdV} with $k=1$ in weighted Sobolev spaces with low regularity is the following pointwise formula, proved by Fonseca, Linares and Ponce in \cite{FLP}:
\begin{enumerate}
\item[(vii)] For $r\in(0,1)$ and $u_0\in Z_{4r,r}$ we have for all $t\in\mathbb R$ and for almost every $x\in \mathbb R$:
\begin{align}
|x|^r [W(t)u_0](x)=W(t) (|x|^r u_0)(x)+W(t)\{\Phi_{t,r} (\widehat {u_0}) \}  \veee (x), \label{FW}
\end{align}
where,
\begin{align}
\| (\Phi_{t,r} (\widehat{u_0})(\xi)) \veee \|_{L^2}\leq C_r (1+|t|) (\| u_0\|_{L^2}+\|D_x^{4r} u_0\|_{L^2}). \label{FW1}
\end{align}
\end{enumerate} 
With respect to the weight $\langle x\rangle:=(1+x^2)^{1/2}$, for $N\in\mathbb N$, we will consider a truncated weight $w_N$ of $\langle x \rangle$, such that $w_N\in C^\infty(\mathbb R)$,
\begin{equation}
w_N(x)= \left\{ \begin{array}{ll}
\langle x\rangle \,\, &\mbox{if $|x|\leq N$}, \\ \\
2N \,\, &\mbox{if $|x|\geq 3N$}, \end{array} \right.
\label{peso}
\end{equation}
$w_N$ is non-decreasing in $|x|$ and for $j\in\mathbb{N}$ and $x\in\mathbb{R}$, the derivatives $w_N^{(j)}$ of order $j$ of $w_N$ satisfy
\begin{equation}
|w_N^{(j)}(x)|\leq\frac{c_j}{w_N^{j-1}(x)}\,,\label{cota}
\end{equation}
where the constant $c_j$ is independent from $N$.

Fonseca and Ponce in \cite{FP} deduced the following interpolation inequality, related to the weights $\langle x \rangle$ and $w_N$.
\begin{lemma}\label{uno} Let $a,b>0$ and $f\in Z_{a,b}\equiv H^a (\mathbb R)\cap L^2(\langle x\rangle^{2b }dx)$. Then for any $\theta\in(0,1)$
\begin{equation}\|J^{\theta a}(\langle x\rangle^{(1-\theta)b}f)\|_{L^2}\leq C\|\langle x\rangle^bf\|_{L^2}^{1-\theta}\|J^af\|_{L^2}^{\theta}, \label{inter}
\end{equation}
where $J^a f:=(1-\partial_x^2) ^{a/2} f$.

Moreover, the inequality \eqref{inter} is still valid with $w_N(x)$ as in \eqref{peso} instead of $\langle x\rangle$ with a constant $C$ independent of $N$.
\end{lemma}

Finally, in our arguments we will use the following standard estimate, concerning the weights $\langle x\rangle$ and $w_N$.
\begin{lemma}\label{dos}
Let $b>0$ and $n\in\mathbb{N}$. Suppose that $J^n(\langle x\rangle^{b}u_0)\in L^2(\mathbb{R})$. Then
\begin{equation}
\|\langle x\rangle^{b}\partial_x^nu_0\|_{L^2}\leq C(b,n)\|J^n(\langle x\rangle^{b}u_0)\|_{L^2}\,.\label{leibniz}
\end{equation}
Moreover, the inequality \eqref{leibniz} is still valid with $w_N(x)$ as in \eqref{peso} instead of $\langle x\rangle$ with a constant $C(b,n)$ independent of $N$.
\end{lemma}

\begin{proof}
The proof follows from induction on $n$ and the Leibniz formula.
\end{proof}

\section{Well-Posedness of the IVP \eqref{KdV} with $k=1$\\(Proof of Theorems \ref{prin1} and \ref{prin2})}\label{k=1}

\subsection{Proof of Theorem \ref{prin1}} We consider two cases.

\textit{First case: $5/16<r<1/2$.}

Proceeding as in \cite{KPVWell1} and \cite{KPVWell}, for $u:\mathbb R\times [0,T]\to\mathbb R$ we define:
\begin{align}
\lambda_1^T(u):=& \max_{[0,T]} \| u(t)\|_{H^{4r}},\label{N1}\\
\lambda_2^T(u):=& \| \partial_x u\|_{L^4_TL^\infty_x}, \label{N2}\\
\lambda_3^T(u):=& \| D_x^{4r} \partial_x u \|_{L^\infty_x L^2_T},\label{N3}\\
\lambda_4^T(u):=& (1+T)^{-\rho} \| u \|_{L^2_xL^\infty_T}, \; \text{with $\rho$ a fixed number such that $\rho> \dfrac 34$}.\label{N4}
\end{align}
Additionally, we introduce
\begin{align}
\lambda_5^T(u):=\|  |x|^r u\|_{L^\infty_T L^2_x}.\label{N5}
\end{align}
Let us consider
\begin{align}
\Lambda^T(u):=\max_{1\leq j\leq 5} \lambda_j^T(u), \label{N}
\end{align}
and
\begin{align}
X_T:=\{u\in C([0,T];H^{4r}(\mathbb R)):\Lambda^T(u)<\infty \}.\label{XT}
\end{align}
Using the linear estimates \eqref{le6}, \eqref{le1} and \eqref{le4}, Kenig, Ponce and Vega in \cite{KPVWell}, showed that for $u_0\in H^{4r} (\mathbb R)$, $T>0$ and $1\leq i\leq 4$
\begin{align}
\lambda_i^T (W(t)u_0)\leq C \| u_0\|_{H^{4r}}.\label{Gi}
\end{align}
On the other hand, from \eqref{FW} and \eqref{FW1}, it follows that, for $t\in[0,T]$,
\begin{align}
\lambda_5^T (W(t)u_0)\leq \| |x|^r u_0\|_{L^2}+C_r (1+T)(\| u_0\|_{L^2}+\| D_x^{4r}u_0 \|_{L^2}).\label{G5}
\end{align}
In consequence, for $u_0\in Z_{4r,r}$, the estimates \eqref{Gi} and \eqref{G5} imply that
\begin{align}
\Lambda^T(W(t)u_0)\leq \| |x|^r u_0 \|_{L^2}+C(1+T) \| u_0\|_{H^{4r}}.\label{LG}
\end{align}
Let us denote by $u:=\Phi (v)\equiv \Phi_{u_0}(v)$ the solution of the linear inhomogeneous IVP
\begin{align}
\left. \begin{array}{r}
\partial_{t}u+\partial_{x}^{5}u+v\partial_{x}v=0\text{ \,}\\
u(0)=u_0
\end{array}
\right\} ,\label{KdV1}
\end{align}
where, $v\in X_T^a:=\{ w\in X_T: \Lambda^T(w)\leq a\}$, for $a>0$.

By Duhamel's formula:

\begin{align}
\Phi(v)(t)\equiv u(t)=W(t)u_0-\int_0^t W(t-t')(v\partial_x v)(t') dt'. \label{DF}
\end{align}
Taking into account that
$$\Lambda^T(u)\leq \Lambda^T(W(t)u_0)+\int_0^T \Lambda^T(W(t-t')(v\partial_x v(t'))) dt',$$
from \eqref{LG} it follows that
\begin{align}
\Lambda^T(u)\leq & \| |x|^r u_0\|_{L^2}+C(1+T) \| u_0\|_{H^{4r}}\notag\\
&+C(1+T)(\| v\partial_x v \|_{L^1_T L^2_x}+ \| D_x^{4r} (v\partial_x v)\|_{L^1_T L^2_x})+\| |x|^r v\partial_x v\|_{L^1_T L^2_x}. \label{LDF}
\end{align}
In \cite{KPVWell1} (see proof of Lemma 4.1) it was proved that
\begin{align}
\| v\partial_x v\|_{L^1_T L^2_x}+\| D^{4r}_x (v\partial_x v)\|_{L^1_T L^2_x}\leq & CT^{1/2} (1+T)^\rho \lambda_4^T (v)\lambda_3^T(v)+CT^{3/4}\lambda_2^T(v) \lambda_1^T(v)\notag\\
&+CT(\lambda_1^T(v))^2\notag\\
\leq & C(T^{1/2}(1+T)^\rho+T^{3/4}+T)(\Lambda^T(v))^2,\label{LDF1}
\end{align}
and let us observe that
\begin{align}
\| |x|^r v\partial_x v\|_{L^1_T L^2_x} \leq & CT^{3/4} \| |x|^r v\partial_x v\|_{L^4_T L^2_x}\leq CT^{3/4} \| |x|^r v\|_{L^\infty_T L^2_x} \| \partial_x v \|_{L^4_T L^\infty_x}\notag\\
\leq & CT^{3/4} \lambda_5^T (v)\lambda_2^T(v)\leq C T^{3/4}(\Lambda^T(v))^2. \label{LDF2}
\end{align}
From \eqref{LDF}-\eqref{LDF2} it follows that
$$\Lambda^T(u)\leq \| |x|^r u_0\|_{L^2} +C (1+T) \|u_0 \|_{H^{4r}}+C(1+T)(T^{1/2}(1+T)^\rho+T^{3/4}+T)(\Lambda^T(v))^2.$$
Taking $a:=2(\| |x|^r u_0\|_{L^2} +C (1+T) \|u_0 \|_{H^{4r}})$ and $T$ sufficiently small in order to have
$$C(1+T)(T^{1/2}(1+T)^\rho+T^{3/4}+T)a<\dfrac 12,$$
it can be seen that $\Phi: X_T^a \to X_T^a$.
Reasoning as in \cite{KPVWell} (proof of Theorem 2.1), for $T>0$ small enough, $\Phi: X_T^a \to X_T^a$ is a contraction. In consequence, there exists a unique $u\in X_T^a$ such that $\Phi(u)=u$. In other words, for $t\in[0,T]$:
$$u(t)=W(t)u_0-\int_0^t W(t-t')(u\partial_x u)(t') dt'.$$
To conclude the proof of this case we reason in the same manner as it was done at the end of the proof of Theorem 2.1 in \cite{KPVWell}.

\textit{Second case: $r\geq1/2$.}

By Theorem \ref{principal} there exist $T=T(\|u_0 \|_{H^{4r}})$ and a unique $u$ in the class defined by the conditions \eqref{con}-\eqref{H} with $s=4r$, which is a solution of the IVP \eqref{KdV} with $k=1$. Let $\{u_{0m}\}_{m\in\mathbb N}$ be a sequence in $C_0^\infty(\mathbb R)$ such that $u_{0m}\to u_0$ in $H^{4r}(\mathbb R)$ and let $u_m\in C([0,T];H^\infty(\mathbb R))$ be a solution of the equation in \eqref{KdV} corresponding to the initial data $u_{0m}$. (Without loss of generality we can suppose that $u_m$ is defined in the same interval $[0,T]$ (see regularity property in Theorem \ref{principal})). By Theorem \ref{principal} $u_m\to u$ in $C([0,T];H^{4r}(\mathbb R))$.\\
We multiply the equation
\begin{equation}
\partial_t u_m +\partial_x^5 u_m +u_m\partial_x u_m=0\label{ecnm}
\end{equation}
by $u_m w_N^{2r}$, where $w_N$ is the truncated weight defined in \eqref{peso}, and for a fixed $t\in[0,T]$, we integrate in $\mathbb R$ with respect to $x$ and use integration by parts to obtain
\begin{align}
\dfrac d{dt} (u_m(t),u_m(t)w_N^{2r})=&5(\partial_x^2u_m(t),\partial_x^2u_m(t)(w_N^{2r})^{(1)})-5(\partial_x u_m(t),\partial_x u_m(t)(w_N^{2r})^{(3)})\notag\\
&+(u_m(t),u_m(t)(w_N^{2r})^{(5)})+\dfrac 23 (1,u_m(t)^3(w_N^{2r})^{(1)}),\label{derum}
\end{align}
where $(\cdot,\cdot)$ denotes the inner product in $L^2(\mathbb R)$.

Integrating last equation with respect to the time variable in the interval $[0,t]$, we have
\begin{align}
(u_m(t),u_m(t)w_N^{2r})=&(u_{0m},u_{0m}w_N^{2r})+5\int_0^t (\partial_x^2 u_m(t'),\partial_x^2 u_m(t')(w_N^{2r})^{(1)})dt'\notag\\
&-5\int_0^t (\partial_xu_m(t'),\partial_x u_m (t')(w_N^{2r})^{(3)})dt'+\int_0^t (u_m(t'),u_m(t')(w_N^{2r})^{(5)})dt'\notag\\
&+\dfrac 23 \int_0^t (1,u_m(t')^3(w_N^{2r})^{(1)})dt'. \label{intum}
\end{align}
Since $u_m\to u$ in $C([0,T]; H^{4r}(\mathbb R))$, with $4r\geq 2$, and the weight $w_N^{2r}$ and its derivatives are bounded functions, it follows from equality \eqref{intum}, after passing to the limit when $m\to\infty$, that
\begin{align}
(u(t),u(t)w_N^{2r})=&(u_{0},u_{0}w_N^{2r})+5\int_0^t (\partial_x^2 u(t'),\partial_x^2 u(t')(w_N^{2r})^{(1)})dt'\notag\\
&-5\int_0^t (\partial_xu(t'),\partial_x u (t')(w_N^{2r})^{(3)})dt'+\int_0^t (u(t'),u(t')(w_N^{2r})^{(5)})dt'\notag\\
&+\dfrac 23 \int_0^t (1,u(t')^3(w_N^{2r})^{(1)}) dt'\notag\\
\equiv&I+II+III+IV+V.\label{intu}
\end{align}
Let us estimate the terms on the  side of \eqref{intu}. First of all
\begin{equation}
I\leq \| u_0 \|_{L^2(\langle x \rangle^{2r}dx)}^2.\label{cotI}
\end{equation}
With respect to the term $II$, using Lemmas \ref{dos} and \ref{uno}, we have
\begin{align}
|II| &\leq 10r \int_0^t (\partial_x^2 u(t'),\partial_x^2 u(t') w_N^{2r-1} |(w_N)^{(1)}|)dt'\notag\\
&\leq C \int_0^t (\partial_x^2 u(t'),\partial_x^2 u(t') w_N^{2r-1})dt'
\leq C \int_0^t \|J^2( w_N ^{r-\frac 12} u(t')) \|_{L^2}^2 dt'\notag\\
&\leq C \int_0^t  \|J^{4r} u(t') \|_{L^2}^{1/r} \| w_N^r u(t') \|_{L^2}^{2-1/r} dt'
\leq C \int_0^t \|w_N^r u(t')\|_{L^2}^{2-1/r}dt'\notag\\
&\leq C\int_0^t (1+\|w_N^r u(t') \|_{L^2}^2) dt'
\leq Ct +C\int_0^t (u(t'),u(t')w_N^{2r})dt'. \label{cotII}
\end{align}
Using inequality \eqref{cota} for the derivatives of $w_N$ it can be seen that
\begin{equation}
|(w_N^{2r})^{(3)}|\leq Cw_N^{2r-3}\quad\text{and}\quad |(w_N^{2r})^{(5)}|\leq Cw_N^{2r-5}\,.
\end{equation}
In this manner we can bound the term III as follows:
\begin{equation}
|III|\leq C\int_0^t (\partial_xu(t'),\partial_x u (t')w_N^{2r-3})dt'\,.
\end{equation}
If $2r-3\leq 0$, since $u\in C([0,T];H^{4r}(\mathbb R))$ with $4r\geq 2$, it is clear that
\begin{equation}
|III|\leq Ct\,.\label{cotaIII}
\end{equation}
If $2r-3>0$, we apply Lemmas \ref{dos} and \ref{uno} to conclude that
\begin{align}
|III|&\leq C \int_0^t \|J( w_N^{r-\frac 32} u(t')) \|_{L^2}^2 dt'\notag\\
&\leq C \int_0^t  \|J^{4r} u(t') \|_{L^2}^{\frac1{2r}} \| w_N^\frac{r(r-3/2)}{(r-1/4)} u(t') \|_{L^2}^{\frac{4r-1}{2r}} dt'\notag\\
&\leq C \int_0^t \|w_N^\frac{r(r-3/2)}{(r-1/4)} u(t')\|_{L^2}^{2-\frac1{2r}}dt'
\leq C \int_0^t \|w_N^{r} u(t')\|_{L^2}^{2-\frac1{2r}}dt'\notag\\
&\leq C\int_0^t (1+\|w_N^r u(t') \|_{L^2}^2) dt'
\leq Ct +C\int_0^t (u(t'),u(t')w_N^{2r})dt'. \label{cotIII}
\end{align} 
In any case the estimate \eqref{cotIII} holds. In a similar manner it can be shown the following estimate for the term IV:
\begin{equation}
|IV|\leq C\int_0^t (u(t'),u(t')w_N^{2r})dt'. \label{cotIV}
\end{equation}
With respect to the term V we have:
\begin{align}
|V|&\leq C\int_0^t\|u(t')\|_{L^{\infty}}(u(t'),u(t')w_N^{2r-1})dt'
\leq C\int_0^t\|u(t')\|_{H^{4r}}(u(t'),u(t')w_N^{2r})dt'\notag\\
&\leq C\int_0^t(u(t'),u(t')w_N^{2r})dt'\,.\label{cotV}
\end{align}
From equality \eqref{intu} and the estimates \eqref{cotI}-\eqref{cotV} it follows that, for $t\in[0,T]$,
\begin{equation}
(u(t),u(t)w_N^{2r})\leq \|u_0\|_{L^2(\langle x\rangle^{2r}dx)}^2+Ct+C\int_0^t(u(t'),u(t')w_N^{2r})dt'\,.\notag
\end{equation}
Gronwall's inequality enables us to conclude that, for $t\in[0,T]$,
\begin{equation}
(u(t),u(t)w_N^{2r})\leq \|u_0\|_{L^2(\langle x\rangle^{2r}dx)}^2+Ct+C\int_0^t(\|u_0\|_{L^2(\langle x\rangle^{2r}dx)}^2+Ct')e^{C(t-t')}dt'\,.\label{Gron}
\end{equation}
Passing to the limit in \eqref{Gron} when $N\to\infty$ we obtain, for $t\in[0,T]$,
\begin{equation}
\|u(t)\|_{L^2(\langle x\rangle^{2r}dx)}^2\leq \|u_0\|_{L^2(\langle x\rangle^{2r}dx)}^2+Ct+C\int_0^t(\|u_0\|_{L^2(\langle x\rangle^{2r}dx)}^2+Ct')e^{C(t-t')}dt'\leq C(T), \label{Gron1}
\end{equation}
which implies that $u\in L^{\infty}([0,T];L^2(\langle x\rangle^{2r}dx))$.

Now let us see that $u\in C([0,T];L^2(\langle x\rangle^{2r}dx))$. For that we follow an argument contained in \cite{CP} and \cite{IO}. From \eqref{Gron1} it is clear that there is a positive constant $M$ such that, for all $t\in[0,T]$,
\begin{equation}
\|u(t)\|_{L_w^2}^2\leq \|u_0\|_{L_w^2}^2+Mt, \label{M}
\end{equation}
where the notation $L_w^2:=L^2(\langle x\rangle^{2r}dx)$ was used. 

Taking into account that $u\in C([0,T];L^2)$ and using \eqref{M}, it can be seen that, for $\phi\in L_w^2$, the function $t\mapsto(\phi,u(t))_{L_w^2}$ is continuous from $[0,T]$ to $\mathbb C$. From this fact and \eqref{M} it follows that
\begin{align}
\|u(t)-u(0)\|_{L_w^2}^2&=\|u(t)\|_{L_w^2}^2+\|u(0)\|_{L_w^2}^2-2\;\text{Re}\;(u(0),u(t))_{L_w^2}\notag\\
&\leq \|u(0)\|_{L_w^2}^2+Mt+\|u(0)\|_{L_w^2}^2-2\;\text{Re}\;(u(0),u(t))_{L_w^2}\to 0\quad\text{when}\;t\to 0^+\,,\notag
\end{align}
which proves that $u:[0,T]\longrightarrow L^2(\langle x\rangle^{2r}dx)$ is continuous at $t=0$.

The continuity of $u$  at a point $t_0\in(0,T]$ is a consequence from the continuity  of $u$ at $t=0$ and from the fact that the functions $v_1(x,t):=u(x,t_0+t)$ and $v_2(x,t):=u(-x,t_0-t)$ are also solutions of the fifth order KdV equation. In this manner, we had proved that if $u_0\in Z_{4r,r}$ $(r\geq\frac12)$ there exist $T=T(\|u_0 \|_{H^{4r}})>0$ and a unique $u\in C([0,T]; Z_{4r,r})$, solution of the IVP \eqref{KdV}, with $k=1$, belonging to the class defined by the conditions \eqref{con}-\eqref{H} with $s=4r$. 

Finally, let us prove that if $\widetilde u_m\in C([0,T]; Z_{4r,r})$ is the solution of the fifth order KdV equation, corresponding to the initial data $\widetilde u_{m0}$, where $\widetilde u_{m0}\to u_0$ in $Z_{4r,r}$ when $m\to\infty$, then $\widetilde u_m\to u$ in $C([0,T]; Z_{4r,r})$. By Theorem \ref{principal} we have that $\widetilde u_m\to u$ in $C([0,T];H^{4r})$. In consequence we only must prove that $\widetilde u_m\to u$ in $C([0,T];L^2(\langle x\rangle^{2r}dx))$. Let $v_m:=\widetilde u_m -u$ and $v_{m0}:=\widetilde u_{m0}-u_0$. Proceeding in a similar manner as it was done when we established that $u\in L^{\infty}([0,T];L^2(\langle x\rangle^{2r}dx))$ and taking into account that $v_m\to 0$ in $C([0,T];H^{4r})$ it can be seen that, for $t\in[0,T]$,
\begin{equation}
\|v_m(t)\|_{L^2(w_N^{2r}dx))}^2\leq \|v_{m0}\|_{L^2(\langle x\rangle^{2r}dx)}^2+C_mt+C\int_0^t\|v_m(t')\|_{L^2(w_N^{2r}dx))}^2dt'\,,\notag
\end{equation}
where $\underset{m\to\infty}\lim C_m=0$. Hence, by Gronwall's inequality, we have for $t\in[0,T]$ and $N\in\mathbb N$ that
\begin{equation}
\|v_m(t)\|_{L^2(w_N^{2r}dx))}^2\leq (\|v_{m0}\|_{L^2(\langle x\rangle^{2r}dx)}^2+C_mT)e^{CT}\,.\notag
\end{equation}
From this inequality it follows, after passing to the limit when $N\to\infty$, that
\[v_m\to 0\quad\text{in}\;C([0,T];L^2(\langle x\rangle^{2r}dx))\,.\]
The proof of Theorem \ref{prin1} is complete. \qed

\subsection{Proof of Theorem \ref{prin2}} Taking into account Remarks \ref{R1}(a) and \ref{R1}(b) it is sufficient to show that the IVP \eqref{KdV} for the fifth order KdV equation is globally well-posed in $H^2(\mathbb R)$.

To see this, first of all, we prove that if $u\in C([0,T];H^2(\mathbb R))$ is a solution of the IVP \eqref{KdV} then, for all $t\in[0,T]$, 
\begin{equation}
\| u(t) \|_{H^2}^2\leq K\equiv K(\|u_0 \|_{H^2}),\label{EstAprio}
\end{equation}
where $K$ only depends on $\| u_0\|_{H^2(\mathbb R)}$.

Let us observe that
\begin{align}
\int_\mathbb R (\partial_x u)^2(t) dx\leq  \dfrac 12  \left[ \int_\mathbb R (\partial_x^2 u)^2(t) dx+ \int_\mathbb R u^2(t)dx\right].\label{des1}
\end{align}

Using the definition of the $H^2$-norm, inequality \eqref{des1} and the conservation laws \eqref{cl1} and \eqref{cl2} it follows that
\begin{align}
\| u(t) \|^2_{H^2}&=\int_\mathbb R u^2(t) dx+ \int_\mathbb R (\partial_x u)^2(t) dx+\int_\mathbb R (\partial^2_x u)^2 (t) dx\notag\\
&\leq \dfrac 32 \int_\mathbb R u^2(t) dx+\dfrac 32 \int_\mathbb R (\partial^2_x u)^2 (t) dx\notag\\
&=\dfrac 32 I_1(t)+3 I_2^1(t)-\dfrac 12 \int_\mathbb R u^3(t) dx\notag\\
&=	\dfrac 32 \| u_0 \|_{L^2}^2+3\left[\dfrac 12 \|\partial_x^2 u_0 \|^2_{L^2}+\dfrac 16 \int_\mathbb R u_0^3 dx \right]-\dfrac 12 \int_\mathbb R u^3(t) dx.\label{des2}
\end{align}
Now, from the Sobolev lemma, we have that
\begin{align}
\int_\mathbb R u_0^3 dx\leq \|u_0 \|_{L^\infty} \int_\mathbb R u_0^2 dx\leq C \|u_0 \|_{H^2}^3\label{des3}.
\end{align}
On the other hand, the Sobolev lemma, the conservation law \eqref{cl1} and Young's inequality imply that
\begin{align}
\left|\int_\mathbb R u^3(t) dx\right|\leq &\|u(t) \|_{L^\infty} \|u(t) \|^2_{L^2}\leq C\|u(t) \|_{H^1} \|u(t) \|_{L^2}^2=C\|u(t) \|_{H^1}\|u_0 \|^2_{L^2}\notag\\
\leq & \dfrac 12 \| u(t)\|^2_{H^1}+\dfrac {C^2}2 \|u_0 \|^4_{L^2}\leq  \dfrac 12 \| u(t)\|^2_{H^1}+\dfrac {C^2}2 \|u_0 \|^4_{H^2}.\label{des4}
\end{align}
Therefore, from \eqref{des2}-\eqref{des4}, we have that
$$\| u(t) \|^2_{H^2}\leq \dfrac 32 \| u_0\|^2_{H^2}+C \|u_0 \|^3_{H^2}+ \dfrac{C^2}4 \| u_0\|^4_{H^2} +\dfrac 14 \|u(t) \|^2_{H^2},$$
and from the above inequality
\begin{equation}
\| u(t) \|^2_{H^2}\leq C (\| u_0\|^2_{H^2}+\|u_0 \|^3_{H^2}+\|u_0 \|^4_{H^2})\equiv K,\label{des5}
\end{equation}
which proves \eqref{EstAprio}.

Now we show how to extend the local solution $u$ to any time interval. From the proof of Theorem \ref{principal} it can be seen that the size of the time interval of the solution $u\in C([0,T];H^2(\mathbb R))$ of the IVP \eqref{KdV} is such that
$$T\geq \min\left\{ 1, \dfrac 1{C\|u_0 \|^2_{H^2}}\right\}.$$
Reasoning as in the proof of Theorem \ref{principal} we obtain a solution $u\in C([T,T+t_0];H^2(\mathbb R))$ of the IVP
\begin{align*}
\left. \begin{array}{rr}
\partial_{t}v+\partial_{x}^{5}v+v\partial_{x}v=0,\text{} & \quad x,t\in \mathbb R\\
v(T)=u(T) &
\end{array}
\right\} ,
\end{align*}
such that
$$t_0\geq \min\left\{ 1, \dfrac 1{C\|u(T)\|^2_{H^2}}\right\}.$$
In this manner we obtain a solution $u\in C([0,T+t_0];H^2(\mathbb R))$ of the IVP \eqref{KdV}.
By the a priori estimate \eqref{EstAprio} we have that
$$\dfrac 1{\|u(t) \|_{H^2}^2}\geq \dfrac 1{K},$$
for $t\in[0,T+t_0]$, and therefore
$$t_0\geq \min\left\{1,\dfrac 1{CK}  \right\}.$$
We repeat this argument $n+1$ times to obtain a solution $u\in C([0,T+t_0+\cdots+t_n];H^2(\mathbb R))$ with
$$t_j\geq \min\left\{ 1, \dfrac 1{CK}\right\}, \; j=0,\dots, n.$$
Since
$$\sum_{j=0}^n t_j\to\infty$$
as $n\to \infty$ then we can extend the solution to any time interval.\qed

\section{Well-Posedness of the IVP \eqref{KdV} with $k=2$\\(Proof of Theorems \ref{tmkdv1} and \ref{tmkdv2})}\label{k=2}

\subsection{Proof of Theorem \ref{tmkdv1}} For $T>0$, let us define the space
\begin{align}
Y_T:=&\{ u\in C([0,T];H^2(\mathbb R)): \|\partial_x^4 u \|_{L^\infty_x L^2_T}<\infty,\, \| u\|_{L^{16/5}_x L^\infty_T}<\infty,\, \| u\|_{L^4_x L^\infty_T}<\infty  \}, \label{mY}
\end{align}
and, for $u\in Y_T$, let us consider the norms:
\begin{align}
\lambda_1^T(u):=& \max_{[0,T]} \| u(t)\|_{H^2},\label{mN1}\\
\lambda_2^T(u):=& \| \partial^4 _x u\|_{L^\infty_x L^2_T}, \label{mN2}\\
\lambda_3^T(u):=& \| u \|_{L^{16/5}_x L^\infty_T},\label{mN3}\\
\lambda_4^T(u):=& \| u \|_{L^4_xL^\infty_T}, \label{mN4}\\
\Lambda^T(u):=&\max_{1\leq i\leq 4} \lambda^T_i (u). \label{mN}
\end{align}
For $a>0$, let $Y_T^a$ be the closed ball in $Y_T$ defined by
\begin{align}
Y_T^a:= \{ u\in Y_T : \Lambda^T(u)\leq a\}. \label{mYa}
\end{align}

We will prove that there exist $T>0$ and $a>0$ such that the operator
\begin{align*}
\Psi: Y_T^a& \to Y_T ^a\\
v& \mapsto \Psi(v)=W(t)u_0-\int_0^t W(t-t') (v^2 \partial_x v)(t') dt'
\end{align*}
is a contraction.

Besides the linear estimates in section \ref{pre}, we will need some nonlinear estimates in order to prove that $\Psi$ is a contraction.

First of all we establish these nonlinear estimates. Let $u\in Y_T$:
\begin{enumerate}
\item[(i)] Using interpolation we have
\begin{align}
\| u^2\partial_x u\|_{L^1_xL^2_T}\leq & \|u^2 \|_{L^{8/5}_x L^\infty_T} \| \partial_x u\|_{L^{8/3}_x L^2_T}\leq  \| u\|^2_{L^{16/5}_x L^\infty_T} \| u\|^{3/4}_{L^2_x L^2_T} \| \partial_x^4 u\|^{1/4}_{L^\infty_x L^2_T}\notag \\
\leq & C T^{3/4} \| u\|^2_{L^{16/5}_x L^\infty _T} \| u \|^{3/4}_{L^\infty_T L^2_x} \| \partial_x^4 u\|^{1/4}_{L^\infty_x L^2_T}. \label{nle1}
\end{align}

\item[(ii)] By \eqref{le2} and \eqref{nle1} it follows that
\begin{align}
\| \partial^4_x \int_0^t W(t-t') (u^2\partial_x u)(t')dt' \|_{L^\infty_x L^2_T}\leq & C\| u^2 \partial_x u\|_{L^1_x L^2_T}\leq C T^{3/4} \| u\|^2_{L^{16/5}_x L^\infty_T} \| u\|^{3/4}_{L^\infty_T L^2_x} \| \partial_x^4 u\|^{1/4}_{L^\infty_x L^2_T}. \label{nle2}
\end{align}

\item[(iii)] By \eqref{le3} and \eqref{nle1} it follows that
\begin{align}
\| \partial^2_x \int_0^t W(t-t') (u^2\partial_x u)(t')dt' \|_{L^\infty_T L^2_x}\leq C \| u^2\partial_x u\|_{L^1_x L^2_T}\leq C T^{3/4} \| u\|^2_{L^{16/5}_x L^\infty_T} \| u\|^{3/4}_{L^\infty_T L^2_x} \| \partial_x^4 u\|^{1/4}_{L^\infty_x L^2_T}. \label{nle3}
\end{align}
\end{enumerate}
Now we prove that there exist $T>0$ and $a>0$ such that $\Psi(Y_T^a)\subset Y_T^a$.

Let $v\in Y_T^a$. Then by \eqref{nle3}
\begin{align}
\notag &\lambda_1^T(\Psi(v))\\
\notag &\leq  \lambda_1^T(W(t)u_0)+\lambda_1^T(\int_0^t W(t-t')(v^2\partial_x v)(t') dt')\\
\notag &\leq \|u_0 \|_{H^2} +C \left( \sup_{[0,T]} \|\int_0^t W(t-t')(v^2\partial_x v)(t') dt' \|_{L^2}+\sup_{[0,T]} \|\partial_x^2 \int_0^t W(t-t')(v^2\partial_x v)(t')dt' \|_{L^2} \right)\\
\notag &\leq \| u_0 \|_{H^2}+CT\sup_{[0,T]} \| v(t) \|^3_{H^2}+CT^{3/4}\|v \|^2_{L_x^{16/5}L_T^\infty} \|v \|^{3/4}_{L^\infty_T L^2_x} \| \partial_x^4 v\|^{1/4}_{L^\infty_x L^2_T}\\
&\leq \|u_0 \|_{H^2}+CT^{3/4}(T^{1/4}+1)(\Lambda^T(v))^3.\label{mdes1}
\end{align}
From \eqref{le1} and \eqref{nle2} it follows that
\begin{align}
\notag \lambda_2^T(\Psi(v)) & \leq \| \partial_x^4 W(t)u_0 \|_{L^\infty_x L^2_T}+\|\partial_x^4 \int_0^t W(t-t') (v^2\partial_x v)(t') dt'\|_{L^\infty_x L^2_T}\\
\notag &\leq C \| \partial_x^2 u_0 \|_{L^2}+CT^{3/4} \| v\|^2_{L^{16/5}_x L^\infty_T} \| v\|^{3/4}_{L^\infty_T L^2_x}\|\partial_x^4 v \|^{1/4}_{L^\infty_x L^2_T}\\
\leq &C \|u_0 \|_{H^2}+CT^{3/4} (\Lambda^T(v))^3.\label{mdes2} 
\end{align}
By using \eqref{le6m}, the Leibniz rule and interpolation, we obtain
\begin{align}
\notag \lambda_3^T(\Psi(v))  \leq &\| W(t)u_0\|_{L^{16/5}_x L^\infty_T}+\|\int_0^t W(t-t')(v^2 \partial_x v)dt' \|_{L^{16/5}_x L^\infty_T}\\
\notag \leq &C(1+T)^\rho \|u_0 \|_{H^2}+C(1+T)^{\rho}\int_0^T \| v^2\partial_x v(t')\|_{H^2} dt'\\
\notag \leq &C(1+T)^\rho \|u_0 \|_{H^2}+C(1+T)^{\rho}\int_0^T \| v^2\partial_x v(t')\|_{L^2} dt'\\
\notag &+C(1+T)^{\rho}\int_0^T \| \partial_x^2 (v^2\partial_x v) (t')\|_{L^2} dt'\\
\notag \leq & C(1+T)^\rho \| u_0\|_{H^2}+C(1+T)^\rho T(\Lambda^T(v))^3+C(1+T)^\rho \left( \int_0^T \| (\partial_x v)^3(t') \|_{L^2} dt' \right.\\
\notag & \left. + \int_0^T \| (v \partial_x v\partial_x^2 v)(t')\|_{L^2} dt' +\int_0^T \| (v^2\partial_x^3 v)(t')\|_{L^2} dt' \right)\\
\notag \leq & C(1+T)^\rho \| u_0\|_{H^2}+C(1+T)^\rho T(\Lambda^T (v))^3+C(1+T)^\rho T^{1/2} \| v^2 \partial_x^3 v \|_{L^2_T L^2_x}\\
\notag \leq & C (1+T)^\rho \| u_0\|_{H^2}+C(1+T)^\rho T(\Lambda^T (v))^3+C(1+T)^\rho T^{1/2} \|v^2 \|_{L_x^{16/7}L^8_T} \| \partial_x^3 v\|_{L_x^{16}L_T^{8/3}}\\
\notag \leq & C  (1+T)^\rho \| u_0\|_{H^2}+C(1+T)^\rho T(\Lambda^T (v))^3\\
\notag &+C(1+T)^\rho T^{1/2} \|v \|_{L^{16/5}_x L^\infty_T} \| v\|_{L^8_x L^8_T} \| v\|^{1/4}_{L^4_x L^\infty_T} \| \partial_x^4 v\|^{3/4}_{L^\infty_x L^2_T}\\
\notag & \leq C  (1+T)^\rho \| u_0\|_{H^2}+C(1+T)^\rho T(\Lambda^T (v))^3\\
\notag &+C(1+T)^\rho T^{1/2} \lambda_3^T (v) T^{1/8} \lambda_1^T (v) \lambda_4^T (v)^{1/4}\lambda_2^T (v)^{3/4}\\
\leq & C  (1+T)^\rho \| u_0\|_{H^2}+C(1+T)^\rho T(\Lambda^T (v))^3+C(1+T)^\rho T^{5/8} (\Lambda^T(v))^3.\label{mdes3}
\end{align}
Applying \eqref{le5} we have that
\begin{align}
\notag \lambda_4^T(\Psi(v))  \leq & \|W(t)u_0 \|_{L^4_x L^\infty_T}+\int_0^T \| W(t-t')(v^2 \partial_x v)(t') \|_{L^4_x L^\infty_T} dt'\\
\notag \leq & C \| D_x^{1/4} u_0 \|_{L^2}+C \int_0^T \| D^{1/4}_x (v^2 \partial_x v(t'))\|_{L^2_x} dt'\\
\notag \leq & C \| D_x^{1/4} u_0 \|_{L^2}+C \int_0^T \| (v^2\partial_x v)(t')\|_{L^2}dt'+C\int_0^T \| \partial_x (v^2 \partial_x v)(t')\|_{L^2} dt'\\
\leq & C \| u_0 \|_{H^2}+CT (\Lambda^T(v))^3. \label{mdes4}
\end{align}

From \eqref{mdes1}-\eqref{mdes4} we obtain
\begin{align}
\Lambda^T(\Psi(v))\leq C(1+T)^\rho \| u_0\|_{H^2}+C T^{5/8} [T^{1/8}(T^{1/4}+1)+(1+T)^\rho(1+T^{3/8})+T^{3/8}] (\Lambda^T(v))^3. \label{mdes5}
\end{align}
Let us take $a:=2C(1+T)^\rho \| u_0\|_{H^2}$ and $T$ in such a way that
\begin{align}
C T^{5/8} [T^{1/8}(T^{1/4}+1)+(1+T)^\rho(1+T^{3/8})+T^{3/8}] a^2\leq \dfrac 12. \label{mdes6}
\end{align}
Since $\Lambda^T(v)\leq a$, from \eqref{mdes5} and \eqref{mdes6}, we have that
$$\Lambda^T(\Psi(v))\leq\dfrac a2+\dfrac 12(\Lambda^T(v))\leq \dfrac a2+\dfrac a2=a,$$
i.e. $\Psi(Y_T^a)\subset Y_T^a$.
Now we will find an additional condition on the size of $T$ in order to guarantee that the operator $\Psi: Y_T^a\to Y_T^a$ is a contraction. Let $v,w\in Y_T^a$. Then
\begin{align*}
\Psi(w)-\Psi(v)=\int_0^t W(t-t')(v^2\partial_x (v-w))(t') dt'+\int_0^t W(t-t')((v+w)(v-w)\partial_x w)(t') dt'.
\end{align*}
Therefore, proceeding as before:
\begin{align}
\notag \lambda_1^T(\Psi(w)-\Psi(v))\leq &\lambda_1^T\left(\int_0^t W(t-t')(v^2\partial_x (v-w))(t') dt'\right)\\
\notag &+\lambda_1^T\left(\int_0^t W(t-t')((v+w)(v-w)\partial_x w)(t') dt'\right)\\
\notag \leq & CT(\lambda_1^T(v))^2\lambda_1^T(v-w)+CT^{3/4}(\lambda_3^T(v))^2 \lambda_1^T(v-w)^{3/4} \lambda_2^T(v-w)^{1/4}\\
\notag &+ CT \lambda_1^T(v+w) \lambda_1^T(v-w)\lambda_1^T(w)\\
\notag &+CT^{3/4}\lambda_3^T(v+w)\lambda_3^T(v-w)\lambda_1^T(w)^{3/4}\lambda_2^T(w)^{1/4}\\
\notag \leq&C[(T+T^{3/4})\Lambda^T(v)^2+T \Lambda^T(v+w)\Lambda^T(w)\\
\notag&+T^{3/4}\Lambda^T(v+w)\Lambda^T(w)]\Lambda^T(v-w)\\
\leq& C(T+T^{3/4})(\Lambda^T(v)^2+\Lambda^T(w)^2)\Lambda^T(v-w),\label{mdes7}
\end{align}
\begin{align}
\notag \lambda_2^T(\Psi(w)-\Psi(v))\leq &CT^{3/4}\|v \|^2_{L^{16/5}_x L^\infty_T}\|v-w \|^{3/4}_{L^\infty_T L^2_x} \| \partial_x^4(v-w)\|^{1/4}_{L^\infty_x L^2_T}\\
\notag &+CT^{3/4}\| v+w\|_{L^{16/5}_x L^\infty_T}\|v-w \|_{L^{16/5}_x L^\infty_T} \| w\|_{L^\infty_T L^2_x}^{3/4}\|\partial_x^4w \|_{L^\infty_x L^2_T}^{1/4}\\
\notag \leq & CT^{3/4}\Lambda^T (v)^2 \Lambda^T(v-w)+CT^{3/4}\Lambda^T(v+w)\Lambda^T(v-w)\Lambda^T(w)\\
\leq & CT^{3/4} (\Lambda^T(v)^2+\Lambda^T(w)^2) \Lambda^T(v-w),\label{mdes8}
\end{align}
\begin{align}
\notag \lambda_3^T(\Psi(w)-\Psi(v))\leq &C(1+T)^\rho T\lambda_1^T(v)^2 \lambda_1^T(v-w)\\
\notag &+C(1+T)^\rho T^{1/2}\|v^2 \|_{L^{16/7}_x L^8_T} \| \partial_x^3 (v-w)\|_{L^{16}_x L^{8/3}_T}\\
\notag&+C(1+T)^\rho T \lambda_1^T(v+w)\lambda_1^T(v-w)\lambda_1^T(w)\\
\notag &+C(1+T)^\rho T^{1/2}\| (v+w)(v-w)\|_{L^{16/7}_x L^8_T}\| \partial_x^3 w\|_{L^{16}_x L^{8/3}_T}\\
\notag \leq& C(1+T)^\rho T\lambda_1^T(v)^2 \lambda_1^T(v-w)\\
\notag &+C(1+T)^\rho T^{1/2} \lambda_3^T(v)T^{1/8}\lambda_1^T(v)\lambda_4^T(v-w)^{1/4}\lambda_2^T(v-w)^{3/4}\\
\notag &+C(1+T)^\rho T\lambda_1^T(v+w)\lambda_1^T(v-w)\lambda_1^T(w)\\
\notag &+C(1+T)^\rho T^{1/2} \lambda_3^T(v+w)T^{1/8}\lambda_1^T(v-w)\lambda_4^T(w)^{1/4}\lambda_2^T(w)^{3/4}\\
\leq &CT^{5/8}(1+T)^\rho(1+T^{3/8})(\Lambda^T(v)^2+\Lambda^T(w)^2)\Lambda^T(v-w),\label{mdes9}
\end{align}
\begin{align}
\notag \lambda_4^T(\Psi(w)-\Psi(v))\leq &CT\lambda_1^T(v)^2\lambda_1^T(v-w)+CT\lambda_1^T(v+w)\lambda_1^T(v-w)\lambda_1^T(w)\\
\leq & CT(\Lambda^T(v)^2+\Lambda^T(w)^2)\Lambda^T(v-w).\label{mdes10}
\end{align}
From \eqref{mdes7} to \eqref{mdes10} it follows that
\begin{align}
\notag \Lambda^T(\Psi(w)-\Psi(v))\leq& C(2T+2T^{3/4}+T^{5/8}(1+T)^\rho(1+T^{3/8}))(\Lambda^T(v)^2+\Lambda^T(w)^2)\Lambda^T(v-w)\\
\leq & C(2T+2T^{3/4}+T^{5/8}(1+T)^\rho(1+T^{3/8})) 2a^2\lambda^T(v-w).
\end{align}
If $a:=2C(1+T)^\rho \| u_0\|_{H^2}$ and $T$ are taken satisfying \eqref{mdes6} and the additional condition
\begin{align*}
C(2T+2T^{3/4}+T^{5/8}(1+T)^\rho(1+T^{3/8})) 2a^2<1,
\end{align*}
then $\Psi:Y_T^a\to Y_T^a$ is a contraction. Hence, there exists a unique $u\in Y_T^a$ such that $\Psi(u)=u$.

From this point, proceeding in a similar way as it was done in \cite{L} we can conclude that, given $u_0\in H^2(\mathbb R)$, there exist $T=T(\|u_0 \|_{H^2})>0$ and a unique $u$, solution of the IVP \eqref{KdV} for the modified fifth order KdV equation ($k=2$), such that
\begin{align}
u\in C([0,T]; H^2(\mathbb R)) \label{uinc}
\end{align}
and $u$ satisfies the conditions \eqref{mdes12}, \eqref{mdes13} and \eqref{mdes14}. Moreover, for any $T'\in(0,T)$ there exists a neighborhood $V_1$ of $u_0$ in $H^2(\mathbb R)$, such that the data-solution map $\tilde u_0\mapsto u_0$ from $V_1$ into the class defined by \eqref{uinc}, \eqref{mdes12}, \eqref{mdes13} and \eqref{mdes14} with $T'$ instead of $T$ is Lipschitz.
If additionally, we have that $u_0\in Z_{2,1/2}$, then reasoning as in the proof of Theorem \ref{prin1} (case $r\geq 1/2$) we obtain that $u\in C([0,T];L^2(\langle x\rangle) dx)$, and that there exists a neighborhood $V$ of $u_0$ in $Z_{2,1/2}$ such that the data-solution map $\tilde u_0\mapsto u_0$ from $V$ into the class defined by \eqref{mdes11} to \eqref{mdes14} with $T'$ instead of $T$ is Lipschitz.\qed

\subsection{Proof of Theorem \ref{tmkdv2}} From Theorem \ref{tmkdv1} to prove that the IVP \eqref{KdV} (with $k=2$) is GWP in $Z_{2,1/2}$, it is sufficient to establish that this IVP is GWP in $H^2(\mathbb R)$. Reasoning as in the proof of Theorem \ref{prin2} it is enough to show that if $u\in C([0,T];H^2(\mathbb R))$ is a solution of the IVP \eqref{KdV} with $k=2$, then for every $t\in [0,T]$
\begin{equation}
\| u(t) \|_{H^2}^2\leq K\equiv K(\|u_0 \|_{H^2}),\label{mEstAprio}
\end{equation}
where $K$ only depends on $\| u_0\|_{H^2(\mathbb R)}$.

From \eqref{des1} and the conservation law \eqref{cl1} it is clear that \eqref{mEstAprio} holds if we prove that, for every $t\in[0,T]$,
$$\| \partial_x^2 u(t)\|^2_{L^2}\leq K\equiv K(\|u_0 \|_{H^2}).$$

By the conservation law \eqref{mcl2} we have that
$$\| \partial_x^2 u(t) \|_{L^2}^2=\dfrac 1{12} \int_\mathbb R u_0^4 dx+\int_\mathbb R (\partial_x^2 u_0)^2 dx-\dfrac 1{12}\int_\mathbb R u^4(t) dx,$$
and, since the last term in the  side of the above equality is non-positive, we obtain that
\begin{align*}
\| \partial_x^2 u(t)\|_{L^2}^2&\leq \dfrac 1{12}\int_\mathbb R u_0^4 dx+\int_\mathbb R (\partial_x^2 u_0)^2 dx\leq \dfrac 1{12} \| u_0\|_{L^\infty}^2 \int_\mathbb R u_0^2 dx +\int_\mathbb R (\partial_x^2 u_0)^2 dx.
\end{align*}
Taking into account that $\| u_0\|_{L^\infty}^2\leq C\|u_0 \|_{H^2}^2$, we have
$$\| \partial_x^2 u(t)\|_{L^2}^2\leq C\| u_0\|_{H^2}^2 \int_\mathbb R u_0^2 dx+\int_\mathbb R (\partial_x^2 u_0)^2 dx\leq C\|u_0 \|_{H^2}^2 \| u_0\|^2_{H^2}+\| u_0\|^2_{H^2}\equiv K(\| u_0\|_{H^2}).$$
\qed

\section{Relation Between Decay and Regularity\\for the Solution of the IVP \eqref{KdV} with $k=2$\\(Proof of Theorem \ref{tmkdv3})}

Let us suppose that $t_0=0$ and let $u_0=u(0)$. For $x\geq 0$ and $N\in\mathbb N$ let us  define $\varphi_{N,\alpha}\in C^5([0,\infty))$ such that
$$\varphi_{N,\alpha}(x)=\left\{ \begin{array}{ll} (1+x^2)^{\alpha+1/2}-1 & \text{ if } x\in[0,N],\\(2N^2)^{\alpha+1/2} & \text{ if } x\geq10N, \end{array}\right.$$
$\varphi^{(1)}_{N,\alpha}(x)\geq 0$ and $|\varphi_{N,\alpha}^{(j)}(x)|\leq C$ for $j=2,3,4,5$, with $C$ independent of $N$.

Let $\phi_N\equiv \phi_{N,\alpha}$ be the odd extension of $\varphi_{N,\alpha}$ to $\mathbb R$.

Since $C_0^\infty(\mathbb R)$ is dense in $Z_{2,1/2}$, there exist a sequence $\{u_{0m} \}_{m\in\mathbb N}$ in $C_0^\infty(\mathbb R)$ such that
\begin{align}
\| u_0-u_{0m}\|_{Z_{2,1/2}}\to 0\label{seq}
\end{align}
as $m\to\infty$.

Let $u_m$ be the solution of the modified fifth order KdV equation such that $u_m(0)=u_{0m}$. By Theorems \ref{tmkdv1} and \ref{tmkdv2} we have that
\begin{align}
\| u_m -u\|_{C([0,T];Z_{2,1/2})} &\to 0, \text{ and } \label{seq2}\\
\|\partial_x^4 u_m-\partial_x^4 u \|_{L^\infty_x L^2_T}  &\to 0, \label{seq3}
\end{align}
as $m\to\infty$.

Since $u_{0m}\in H^s(\mathbb R)$ for each $s\in\mathbb R$, it can be seen (regularity property), that $u_m(t)\in H^s(\mathbb R)$ for each $s\in\mathbb R$ and each $t\in[0,T]$.

We now multiply the equation $\partial_t u_m+\partial_x^5 u_m+u_m^2 \partial_x u_m=0$ by $u_m \phi_N$, integrate in $x$ over $\mathbb R$ and apply integration by parts to obtain
\begin{align}
\dfrac d{dt} \int_\mathbb R u_m^2 \phi_N dx- 5\int_\mathbb R (\partial^2_x u_m)^2 \phi_N^{(1)} dx+5\int_\mathbb R (\partial_x u_m)^2 \phi_N^{(3)}dx-\int_\mathbb R u_m^2 \phi_N^{(5)}dx-\dfrac 12 \int_\mathbb R u_m^4 \phi_N^{(1)} dx=0. \label{mdes15}
\end{align}
(In the above equation we use the notation $u_m\equiv u_m(t)$).

From convergence in \eqref{seq2}, since $\alpha\leq 1/2$, it is clear that, for $t\in[0,T]$,
\begin{align}
\notag \left|- \dfrac 12 \int_\mathbb R u_m^4(t) \phi_N^{(1)} dx\right|&\leq C\| u_m(t)\|^2_{L^\infty} \int_\mathbb R u_m^2(t) \langle x \rangle^{2\alpha}dx\leq C\| u_m(t)\|^2_{H^2} \int_\mathbb R u_m^2(t) \langle x \rangle^{2\alpha}dx\\
&\leq C\sup_{t\in[0,T]}\| u_m(t)\|^2_{H^2} \| u_m(t)\|^2_{L^2(\langle x\rangle dx)}\leq C. \label{mdes16}
\end{align}
On the other hand, it is also clear that
\begin{align}
\left|-\int_\mathbb R u_m^2 \phi_N^{(5)}dx\right|\leq C\text{ and } \left|5\int_\mathbb R (\partial_x u_m)^2 \phi_N^{(3)}dx\right|\leq C. \label{mdes17}
\end{align}
Integrating \eqref{mdes15} in $t$ over the interval $[0,t_1]$ and taking into account the inequalities \eqref{mdes16} and \eqref{mdes17} we can conclude that
$$5\int_0^{t_1} \int_\mathbb R (\partial_x^2 u_m)^2 \phi_N^{(1)} dxdt\leq \| u_m^2(t_1)\phi_N\|_{L^1}+\|u_m^2(0) \phi_N \|_{L^1}+Ct_1.$$
Hence
\begin{align}
\limsup_{m\to \infty}\int_0^{t_1} \int_\mathbb R (\partial_x^2 u_m)^2 \phi_N^{(1)} dxdt\leq \dfrac 15 \| u^2(t_1)\phi_N\|_{L^1}+\dfrac 15\|u^2(0) \phi_N \|_{L^1}+Ct_1\leq M, \label{mdes18}
\end{align}
where $M=M(\| \langle x \rangle^{\frac 12 +\alpha}u(0)\|_{L^2}+\| \langle x \rangle^{\frac 12 +\alpha}u(t_1)\|_{L^2})$.

Taking into account that $\phi_N^{(1)}$ is a bounded function, the convergence in \eqref{seq2} implies that
$$\int_0^{t_1}\int_\mathbb R (\partial_x^2 u_m)^2 \phi_N^{(1)} dx dt \to \int_0^{t_1}\int_\mathbb R (\partial_x^2 u)^2\phi_N^{(1)} dx dt,$$
as $m\to\infty$. Therefore, from \eqref{mdes18}, we obtain that
\begin{align}
\int_0^{t_1}\int_\mathbb R (\partial_x^2 u)^2 \phi_N^{(1)} dx dt\leq M. \label{mdes19}
\end{align}
Since $\phi_N^{(1)}$ is an even function, $\phi_N^{(1)}(x)\geq 0$ and, for $x>1$,
$$\phi_N^{(1)}(x)\to (2\alpha+1)(1+x^2)^{\alpha-\frac 12}x\sim \langle x \rangle^{2\alpha}, $$
as $N\to\infty$, applying Fatou's Lemma in \eqref{mdes19}, we have that
$$\int_0^{t_1}\int_{|x|\geq 1} (\partial_x^2 u)^2\langle x \rangle^{2\alpha} dx dt\leq CM,$$
and taking into account that
$$\int_0^{t_1} \int_{|x|\leq 1} (\partial_x^2 u)^2 \langle x \rangle^{2\alpha} dx dt\leq C,$$
we obtain that
\begin{align}
\int_0^{t_1} \int_\mathbb R (\partial_x^2 u)^2 \langle x\rangle^{2\alpha} dx dt\leq C+CM<\infty.\label{mdes20}
\end{align}
From \eqref{mdes20} it follows that
\begin{align}
\partial_x^2 u(t)\in L^2(\langle x \rangle^{2\alpha}dx),\; \text{a.e. } t\in[t_0,t_1]. \label{mdes21}
\end{align}
Let us define, for $t\in[t_0,t_1]$, $w(t):=\partial_x^2 u(t)$ and $w_m(t):=\partial_x^2 u_m(t)$. Then we have that
\begin{align}
\partial_t w_m+\partial_x^5 w_m+u_m^2\partial_x w_m+6 u_m\partial_x u_m w_m+2(\partial_x u_m)^3=0.\label{mdes22}
\end{align}
For $x\geq 0$ and $N\in\mathbb N$ let us define $\tilde \varphi_{N,\alpha}\in C^5([0,\infty))$ such that
$$\tilde\varphi_{N,\alpha}(x)=\left\{\begin{array}{ll} (1+x^2)^\alpha-1 & \text{if } x\in[0,N],\\ (2N^2)^\alpha & \text{if } x\geq 10N, \end{array} \right.$$
$\tilde\varphi_{N,\alpha}^{(1)}(x)\geq 0$ and $|\tilde\varphi_{N,\alpha}^{(j)}(x)|\leq C$ for $j=1,\dots, 5$, with $C$ independent of $N$.

Let $\tilde\phi_N\equiv \tilde\phi_{N,a}$ be the odd extension of $\tilde\varphi_{N,\alpha}$ to $\mathbb R$. Multiplying equation \eqref{mdes22} by $w_m\tilde\phi_N$, integrating in $x$ over $\mathbb R$, using integration by parts, and then integrating in $t$ over an interval $[t_0^*,t_1^*]\subset [t_0,t_1]$ such that $\partial_x^2 u(t_0^*), \partial_x^2 u(t_1^*)\in L^2(\langle x\rangle^{2\alpha} dx)$, we obtain
\begin{align}
\notag \int_{t_0^*}^{t_1^*} \int_\mathbb R (\partial_x^2 w_m)^2 \tilde\phi_N^{(1)}dx dt=&\dfrac 15\int_\mathbb R w_m^2(t_1^*) \tilde\phi_N dx-\dfrac 15 \int_\mathbb R w_m^2(t_0^*) \tilde\phi_N dx\\
\notag&+\int_{t_0^*}^{t_1^*}\int_\mathbb R (\partial_x w_m)^2(t) \tilde\phi_N^{(3)} dx dt-\dfrac 15 \int_{t_0^*}^{t_1^*}\int_\mathbb R w_m^2(t) \tilde\phi_N^{(5)} dx dt\\
\notag&-\dfrac 15 \int_{t_0^*}^{t_1^*} \int_\mathbb R u_m^2(t) w_m^2(t) \tilde\phi_N^{(1)} dx dt+2\int_{t_0^*}^{t_1^*}\int_\mathbb R u_m(t)\partial_x u_m(t) w_m^2(t) \tilde\phi_N dx dt\\
&+\dfrac 45 \int_{t_0^*}^{t_1^*} \int_\mathbb R (\partial_x u_m)^3(t) w_m(t) \tilde\phi_N dx dt.\label{mdes23}
\end{align}
From \eqref{mdes23} we will prove that
$$\int_{t_0^*}^{t_1^*}\int_\mathbb R (\partial_x^2 w)^2(t) \langle x \rangle^{2\alpha-1}dx dt <\infty.$$
Let us observe that
$$\int_{t_0^*}^{t_1^*}\int_\mathbb R (\partial_x w_m)^2(t) \tilde\phi_N^{(3)} dx dt=-\int_{t_0^*}^{t_1^*} \int_\mathbb R w_m(t)\partial_x^2 w_m(t) \tilde\phi_N^{(3)}dx dt+\dfrac 12 \int_{t_0^*}^{t_1^*}\int_\mathbb R w_m^2(t)\tilde\phi_N^{(5)}dx dt.$$
Let $K$ be a constant independent of $N$ such that $|\tilde\phi_N^{(3)}|\leq K\tilde\phi_N^{(1)}$. Then
\begin{align}
\notag \int_{t_0^*}^{t_1^*} \int_\mathbb R(\partial_x w_m)^2(t) \tilde\phi_N^{(3)} dx dt \leq &\dfrac 12 \int_{t_0^*}^{t_1^*} \int_\mathbb R(K w_m^2(t)+\dfrac{(\partial_x^2 w_m)^2(t)}K)|\tilde\phi_N^{(3)}|dx dt\\
\notag &+\dfrac 12 \int_{t_0^*}^{t_1^*} \int_\mathbb R w_m^2(t) \tilde\phi_N^{(5)} dx dt\\
\notag \leq & \dfrac K2 \int_{t_0^*}^{t_1^*} \int_\mathbb R w_m^2(t) |\tilde\phi_N^{(3)}|dx dt+\dfrac 12 \int_{t_0^*}^{t_1^*} \int_\mathbb R(\partial_x^2 w_m)^2(t) \tilde\phi_N^{(1)}dx dt\\
&+\dfrac 12 \int_{t_0^*}^{t_1^*} \int_\mathbb R w_m^2(t) \tilde\phi_N^{(5)} dx dt.\label{mdes24}
\end{align}
Taking into account that the fifth term of the  side of \eqref{mdes23} is not positive, from \eqref{mdes23} and \eqref{mdes24} it follows that
\begin{align}
\notag \dfrac 12 \int_{t_0^*}^{t_1^*} \int_\mathbb R(\partial_x^2 w_m)^2(t) \tilde\phi_N^{(1)} dx dt \leq & \dfrac 15 \int_\mathbb R w_m^2(t_1^*) \tilde\phi_N dx-\dfrac 15 \int_\mathbb R w_m^2(t_0^*)\tilde\phi_N dx\\
\notag &+\dfrac K2 \int_{t_0^*}^{t_1^*} \int_\mathbb R w_m^2(t) |\tilde\phi_N^{(3)}| dx dt+\dfrac 3{10} \int_{t_0^*}^{t_1^*} \int_\mathbb Rw_m^2(t) \tilde\phi_N^{(5)} dx dt\\
\notag &+2 \int_{t_0^*}^{t_1^*} \int_\mathbb R u_m(t) \partial_x u_m(t) w_m^2(t) \tilde\phi_N dx dt\\
\notag &+\dfrac 45\int_{t_0^*}^{t_1^*} \int_\mathbb R(\partial_x u_m)^3(t) w_m(t) \tilde\phi_N dx dt\\
\equiv &\dfrac 15 \int_\mathbb R w_m^2(t_1^*) dx-\dfrac 15 \int_\mathbb R w_m^2(t_0^*)\tilde\phi_N dx+I+II+III+IV. \label{mdes25}
\end{align}
Since $\tilde\phi_N^{(3)}$ and $\tilde\phi_N^{(5)}$ are bounded functions, the convergence in \eqref{seq2} implies that
\begin{align}
I+II\leq C.\label{mdes26}
\end{align}
We now estimate $III+IV$. Using the convergence \eqref{seq2}, and the boundedness of the function $\tilde\phi_N^{(1)}$, we obtain, for $\alpha\in(0,1/4]$, that
\begin{align}
\notag III+IV\leq &C\int_{t_0^*}^{t_1^*} \| u_m(t)\|_{H^2}^3 \|u_m(t) \tilde\phi_N\|_{L^\infty} dt+C\int_{t_0^*}^{t_1^*}\| u_m(t)\|_{H^2}^2\int_\mathbb R |w_m(t)| |\partial_x u_m(t)| |\tilde\phi_N| dx dt\\
\notag \leq& C \int_{t_0^*}^{t_1^*} \| u_m(t) \tilde\phi_N \|_{H^1} dt +C\int_{t_0^*}^{t_1^*} \|w_m(t) \|_{L^2}\|\partial_x u_m(t) \tilde\phi_N  \|_{L^2}dt\\
\notag \leq & C\int_{t_0^*}^{t_1^*} (\| u_m(t) \tilde\phi_N \|_{H^1}+\| \partial_x u_m(t) \tilde\phi_N \|_{L^2})dt\\
\notag \leq &C \int_{t_0^*}^{t_1^*} (\| u_m(t) \tilde\phi_N\|_{L^2}+\| \partial_x u_m(t) \tilde\phi_N \|_{L^2}+\| u_m(t) \tilde\phi_N^{(1)}\|_{L^2}) dt\\
\notag \leq &C \int_{t_0^*}^{t_1^*} (\| u_m(t) \|_{L^2(\langle x\rangle^{4\alpha}dx)}+\| \partial_x u_m(t) \|_{L^2(\langle x\rangle^{4\alpha}dx)}) dt+C\\
\leq & C+C\int_{t_0^*}^{t_1^*} \| \partial_x u_m(t) \|_{L^2(\langle x\rangle^{4\alpha}dx)}dt. \label{mdes27}
\end{align}
Using Lemmas \ref{dos} and \ref{uno} and the convergence in \eqref{seq2} it follows, for $\alpha\in(0,1/8]$ and $t\in[t_0^*,t_1^*]\subset[0,T]$, that
\begin{align}
\notag\| \partial_x u_m(t)\|_{L^2(\langle x \rangle^{4\alpha}dx)}=&\|\langle x \rangle^{2\alpha} \partial_x u_m(t) \|_{L^2}\leq \|\langle x\rangle^{1/4} \partial_x u_m(t) \|_{L^2} \leq C\| J(\langle x \rangle^{1/4} u_m(t))\|_{L^2}\\
\leq & C\| \langle x \rangle^{1/2} u_m(t)\|_{L^2}^{1/2}\| J^2 u_m(t)\|_{L^2}^{1/2} \leq C.\label{mdes28}
\end{align}
From \eqref{mdes25}-\eqref{mdes28} we conclude that, if $\alpha\in(0,1/8]$ then
\begin{align}
\dfrac 12 \int_{t_0^*}^{t_1^*} \int_\mathbb R (\partial_x^2 w_m)^2(t) \tilde\phi_N^{(1)} dx dt \leq \dfrac 15 \int_\mathbb R w_m^2(t_1^*)\tilde\phi_N-\dfrac 15 \int_\mathbb R w_m^2(t_0^*)\tilde\phi_N +C, \label{mdes29}
\end{align}
where $C=C(T)$ is a constant indepent of $N$ and $m$.

Let us observe that from the convergence in \eqref{seq3} we can conclude that
$$\sup_{x\in\mathbb R}\int_0^T |\partial_x^2 w_m(x,t)-\partial_x^2 w(x,t)|^2dt=\| \partial_x^4 u_m-\partial_x^4 u\|_{L^\infty_x L^2_T}^2\to 0,$$
as $m\to\infty$.

Hence, denoting by $|A|$ the measure of a set $A\subset \mathbb R$, we have
\begin{align*}
\Big|\int_{t_0^*}^{t_1^*}\int_\mathbb R[(\partial_x^2 w_m)^2(x,t)&-(\partial_x^2 w)^2(x,t)] \tilde \phi_N^{(1)} dx dt \Big| \\
 &\leq C \int_{supp\, \tilde\phi_N^{(1)}} \int_{t_0^*}^{t_1^*} |\partial_x^2 w_m-\partial_x^2 w| |\partial_x^2 w_m+\partial_x^2 w| dt dx\\
 &\leq C \int _{supp\, \tilde\phi_N^{(1)}} \| \partial_x^2 w_m(x,\cdot)-\partial_x^2 w(x,\cdot)\|_{L^2_T}\| \partial_x^2 w_m(x,\cdot)+\partial_x^2 w(x,\cdot)\|_{L^2_T }dx\\
 &\leq C \| \partial_x^2 w_m-\partial_x^2 w\|_{L^\infty_x L^2_T} \| \partial_x^2 w_m +\partial_x^2 w\|_{L^\infty_x L^2_T} |supp\, \tilde\phi_N^{(1)}| \\
 &\leq C \| \partial_x^2 w_m-\partial_x^2 w\|_{L^\infty_x L^2_T}\to 0,
\end{align*}
as $m\to \infty$.

Therefore, passing to the limit in \eqref{mdes29}, as $m\to\infty$, we obtain that
\begin{align}
\notag \dfrac 12 \int_{t_0^*}^{t_1^*} \int_\mathbb R (\partial_x^2 w)^2(t) \tilde\phi_N ^{(1)} dx dt &\leq \dfrac 15 \int_\mathbb R w^2(t_1^*) \tilde\phi_N -\dfrac 15 \int_\mathbb R w^2(t_0^*) \tilde\phi_N +C\\
&\leq C\int_\mathbb R w^2(t_1^*) \langle x\rangle^{2\alpha} dx+C\int_\mathbb R w^2(t_0^*)\langle x \rangle ^{2\alpha} +C\equiv M.\label{mdes30} 
\end{align}
Since $\tilde\phi_N^{(1)}$ is an even function, $\tilde\phi_N^{(1)}\geq 0$ and, for $x\geq 1$,
$$\tilde\phi_N^{(1)}(x)\to 2\alpha \langle x \rangle^{2\alpha-1} \langle x \rangle' \sim \langle x \rangle^{2\alpha-1},$$
as $N\to\infty$, applying Fatou's Lemma in \eqref{mdes30} we can conclude that
$$\int_{t_0^*}^{t_1^*} \int_{|x|\geq 1} (\partial_x^2 w)^2(t)\langle x\rangle^{2\alpha-1} dx dt\leq CM.$$
On the other hand, since $2\alpha-1\leq 0$,
$$\int_{t_0^*}^{t_1^*} \int_{|x|\leq 1} (\partial_x^2 w)^2(t)\langle x\rangle^{2\alpha-1} dx dt\leq  \int_{t_0^*}^{t_1^*} \int_{|x|\leq 1} (\partial_x^2 w)^2(t) dx dt\leq 2 \| \partial_x^2 w\|^2_{L^\infty_x L^2_T}=2\|\partial_x^4 u \|^2_{L^\infty_x L^2_T}<\infty.$$
In consequence
\begin{align}
\int_{t_0^*}^{t_1^*} \int_\mathbb R (\partial_x^2 w)^2(t) \langle x \rangle^{2\alpha-1} dx dt <\infty.\label{mdes31}
\end{align}
From \eqref{mdes31} it follows that
$$\partial_x^2 w(t) \langle x \rangle ^{\alpha-1/2}\in L^2(\mathbb R), \text{ a.e. } t\in[t_0^*,t_1^*].$$
This fact and \eqref{mdes21} imply that
$$\langle x \rangle^\alpha w(t) \in L^2 \text{ and } J^2(\langle x\rangle^{\alpha-1/2} w(t))\in L^2, \text{ a.e. } t\in[t_0^*,t_1^*].$$
Let us define $f:=\langle x \rangle^{\alpha-1/2} w(t)$. Then
$$\langle x \rangle^{1/2} f\in L^2\text{ and } J^2(f)\in L^2.$$
Using Lemma \ref{uno} with $a=2$ and $b=1/2$ we conclude that, for $\theta\in[0,1]$,
$$\|J^{2\theta}(\langle x\rangle^{(1-\theta)/2} f) \|_{L^2}\leq C\| \langle x \rangle^{1/2} f\|_{L^2}^{(1-\theta)}\|J^2(f) \|^\theta_{L^2}.$$
Taking $\theta=2\alpha$, we have that $\| J^{4\alpha} w(t)\|_{L^2}<\infty$, a.e. $t\in [t_0^*,t_1^*],$ i.e., $w(t)\in H^{4\alpha}$, a.e. $t\in[t_0^*,t_1^*]$, i.e., $u(t)\in H^{2+4\alpha}$.

The end of the proof is as in \cite{ILP} and in consequence we omit the details.\qed

\end{document}